\documentclass[12pt]{article}
\usepackage{amsthm}
\usepackage{amsmath}
\usepackage{cite}
\usepackage{tikz}
\usetikzlibrary{arrows}

\usepackage{amssymb}

\usepackage[margin=1.3in]{geometry}

\usepackage{float}

\newtheorem{theorem}{Theorem}[section]

\newtheorem{proposition}[theorem]{Proposition}
\newtheorem{corollary}[theorem]{Corollary}
\newtheorem{lemma}[theorem]{Lemma}

\newtheorem{definition}[theorem]{Definition}

\newtheorem{remark}[theorem]{Remark}

\newtheorem{question}[theorem]{Question}
\newcommand{\C}{\mathbb{C}}
\newcommand{\Z}{\mathbb{Z}}

\newcommand{\PP}{\mathbb{P}}
\newcommand{\R}{\mathbb{R}}

\newcommand{\M}{M_{\cal T}}
\newcommand{\OM}{\omega_{\lambda_1,\lambda_2}}

\include{epsf}

\begin{document}

\def\CP{\mathbb{C}{\rm P}}
\def\tr{{\rm tr}\,}
\def\endproof{{$\Box$}}
\def\stab{{\rm stab}}

\title{Symplectic and K\"ahler structures on $\mathbb CP^1$-bundles over $\mathbb CP^2$}
\date{\today}
\author{Nicholas Lindsay and Dmitri Panov }

\maketitle

\begin{abstract}

We show that there exist symplectic structures on a $\mathbb CP^1$-bundle over $\mathbb CP^2$ that do not admit a compatible K\"ahler structure. These symplectic structures were originally constructed by Tolman  and they have a Hamiltonian $\mathbb T^2$-symmetry.  Tolman's manifold was shown to be diffeomorphic to a $\mathbb CP^1$-bundle over $\mathbb CP^{2}$ by Goertsches, Konstantis, and Zoller. 
The proof of our result relies on Mori theory, and on classical facts about holomorphic vector bundles over $\mathbb CP^{2}$.

\end{abstract}

\tableofcontents
   
\section{Introduction}

In  \cite{To1} Tolman constructed a remarkable family of symplectic forms $\OM$ on a compact $6$-manifold $\M$, that have a Hamiltonian $\mathbb{T}^2$-action, and yet don't admit a $\mathbb{T}^2$-invariant compatible K\"ahler metric\footnote{We recall Tolman's construction in Section \ref{sec:tolconstruction}, where we also give the definition of Tolman's manifold $\M$, see Definition \ref{def:tolman}.}. Recently, in \cite{GKZ,GKZ2} Goertsches, Konstantis, and Zoller proved that Tolman's manifold $\M$ is diffeomorphic to the projectivisation $\PP(E)$ of a complex rank two bundle $E$ over $\mathbb CP^2$, and hence it admits some K\"ahler structure. The bundle $E$ has $c_1(E)=-1$ and $c_2(E)=-1$, and using such a description, one can restate Tolman's theorem in a more explicit way. 

\begin{theorem}[\!\!\!\cite{To1}, \cite{GKZ}] \label{twosymplectic} Let $E$ be a complex rank two-bundle over $\mathbb CP^2$ with $c_1(E)=-1$, $c_2(E)=-1$, and let let $p:\PP(E) \rightarrow \C\PP^2$ be the projection. Denote by $x$ the generator of $H^2(\mathbb CP^2,\mathbb Z)$. Then for a certain effective $\mathbb T^2$-action on $\mathbb{P}(E)$ and  for any positive $0<\lambda_1< \lambda_2$ there is a $\mathbb T^2$-equivariant symplectic form $\OM$ on $\mathbb{P}(E)$ without a compatible $\mathbb T^2$-invariant K\"ahler metric, and such that 

$$[\OM]= \lambda_1c_1(\mathcal{O}_{\mathbb{P}(E)}(1))+\lambda_2 p^*(x)\in H^2(\mathbb{P}(E),\mathbb R).$$
\end{theorem}

Throughout this article we will denote $c_1(\mathcal{O}_{\mathbb{P}(E)}(1))$ by $\xi$ and $p^*(x)$ by $\eta$. We recall that a symplectic form defines an almost complex structure up to homotopy on the manifold, and in Theorem \ref{twosymplectic} we have $c_1(\M, \OM)=2\xi+2\eta$. Figure \ref{fig:TolmanFirst} shows how the form $\OM$ evaluates on the invariant spheres in $\M$.

To put Theorem \ref{twosymplectic} into context, let us recall a few results on  Hamitonian torus actions, where the symplectic category coincides with the K\"ahler one. First, in \cite{De} Delzant proved that any symplectic manifold $(M^{2n},\omega)$ with an effective Hamiltonian $\mathbb T^n$-action admits a compatible $\mathbb T^n$-invariant K\"ahler structure. Next, a theorem of Karshon \cite{Ka} states that all $4$-dimensional Hamiltonian $S^1$-manifolds admit a compatible $S^1$-invariant K\"ahler form. Furthermore, Tolman \cite{To2} and McDuff \cite{Mc1} proved that any six-dimensional Hamiltonian $S^1$-manifold with $b_2=1$ admits a compatible $S^1$-invariant K\"ahler form. We note finally that in a recent series of papers \cite{Cho1, Cho2, Cho3} Cho proved that all compact monotone symplectic $6$-manifolds with a semi-free Hamiltonian circle action admit a compatible invariant K\"ahler form.

In the view of \cite{De, Ka, To2, Mc1} in several respects Tolman's manifold is a minimal possible example of a Hamiltonian $\mathbb T^k$-manifold without a compatible $\mathbb T^k$-invariant K\"ahler form. Note also, that building on the example of Tolman, Woodward constructed an example of a Hamiltonian, mulitiplicity free, non-K\"ahler action \cite{W}.

In the light of papers \cite{To1, GKZ} it is  natural to ask the following questions.

\begin{question} \label{refinedquestion}

\begin{enumerate} Let $\M$ be Tolman's manifold with the symplectic form $\OM$. 

\item[(a)] Does $(\M,\OM)$ admit a K\"ahler form compatible with the symplectic form $\OM$ for some $0<\lambda_1<\lambda_2$?

\item[(b)] Is there a subgroup $S^1\subset \mathbb{T}^2$ for which  Tolman's manifold admits an $S^1$-invariant K\"ahler form compatible with $\OM$ for some $0<\lambda_1<\lambda_2$.
\end{enumerate}

\end{question}

Our first main result gives a partial answer to Question \ref{refinedquestion} (a).
\begin{theorem} \label{newmain} For  $\lambda_{2} \leq 2 \lambda_{1}$ the symplectic form $\OM$ doesn't admit any compatible K\"ahler metric. 
\end{theorem}

We conjecture, however, that in cases $\frac{\lambda_1}{\lambda_2}$ is small enough, the symplectic form $\OM$ on $\M$ does admit a compatible K\"ahler form\footnote{It would be interesting to see whether the K\"ahler form  discussed in Theorem 4.9 from the published version of \cite{GKZ} could give a positive solution to this conjecture.}.  Theorem \ref{newmain} implies that the symplectic form constructed by Tolman in \cite[Lemma 4.1]{To1} does not admit a compatible K\"{a}hler metric, since this symplectic form is $\omega_{1,2}$ in our notation.

Theorem \ref{newmain} has an immediate corollary, that is new according to our knowledge.

\begin{corollary}\label{firstquestion} There exists a compact  symplectic manifold with a Hamiltonian $S^{1}$-action with isolated fixed points that does not admit a compatible K\"ahler metric.
\end{corollary}

We note that in Theorem \ref{newmain} and Corollary \ref{firstquestion} we do not assume that the compatible K\"{a}hler metric is preserved by any symmetry.

{\it Coprime actions.} To state our second main result, we introduce the following terminology. We say that an $S^1$-action on a symplectic manifold is {\it coprime} if at any fixed point $p$ the weights of the $S^1$-action at $p$ are {\it pairwise coprime} and 
no weight of the action is equal to either $-1,\, 0,\,1$. According to this definition, a coprime action has isolated fixed points. It turns out, that for a subgroup $S^1\subset \mathbb T^2$ the action on $\M$ is coprime if an only if it doesn't have weights $-1,0,1$ at any point. Furthermore, for a generic subgroup $S^1\subset \mathbb T^2$ the action is coprime. Hence, the following result gives an almost complete answer to Question \ref{refinedquestion} (b).

\begin{theorem}\label{main}
Let $(\M,\OM)$ be Tolman's manifold, and let $S^1\subset \mathbb T^2$ be a subgroup whose action on $\M$ is coprime.  Then there is no $S^1$-invariant K\"ahler metric on $\M$ that is compatible with $\OM$. 
\end{theorem}

We believe that one can remove the condition on $S^1$-action to be coprime, and the exactly same statement holds for any subgroup $S^1\subset \mathbb T^2$. 

Let us now explain how  Theorems \ref{twosymplectic}, \ref{newmain} and \ref{main} are proven. As for Theorem \ref{twosymplectic}, using the existence of symplectic forms $\OM$ on $\M$ together with the Duistermaat-Heckman formula, one calculates the cubic intersection form on $H^2(\M,\mathbb Z)$.  Using further localisation one calculates $p_1(\M)$ and $w_2(\M)$. This permits one to apply a theorem of Jupp to get a diffeomorphism between $\M$ and the projectivization of a rank two bundle $E$ over $\mathbb CP^2$. Our proof is similar to \cite{GKZ}, but the approach to calculating the invariants is different. As for the cohomology classes of  the forms $\OM$, they are calculated by evaluating them at the collection of $9$ spheres in $\M$ which are $\mathbb T^2$-invariant.

The proof of Theorem \ref{newmain} uses  Mori theory  together with classical results on holomorphic rank two bundles over $\mathbb CP^2$. Before explaining it, we need to fix some terminology and state an additional result. We will say that two almost complex manifolds $(M,J)$ and $(M',J')$ are {\it almost complex equivalent} if there is a diffeomorphism $\varphi: M\to M'$ such that the almost complex structures $\varphi(J)$ and $J'$ are homotopic on $M'$. Let $J_{\cal T}$ be any almost complex structure on $\M$ that tames $\OM$. Theorem \ref{twosymplectic} states among other things that Tolman's manifold $(\M, J_{\cal T})$ is almost complex equivalent  to the projectivisation of a complex rank two bundle $E$ with $c_1(E)=-1$, $c_2(E)=-1$. The next theorem shows that there is control on K\"ahler metrics of such almost complex manifolds. 

\begin{theorem}\label{kahlerbundle}
Let $(M',J')$ be a compact three-dimensional K\"ahler manifold that is almost complex equivalent to the projectivisation $\mathbb{P}(V)$ of a holomorphic rank two bundle $V$ over $\mathbb CP^2$ with $c_1(V)=-1$. Suppose $4c_2(V)-c_1^2(V)\le 27$. Then $(M', J')$ is biholomorphic to the projecitivisation of a holomorphic bundle $V'$ with $c(V')=c(V)$.
\end{theorem}

It is worth to note here that $c_1^3(\mathbb{P}(V))=2(27+c_1^2(V)-4c_2(V))$. In particular, the quantity $4c_2(V)-c_1^2(V)$ doesn't change when we tensor $V$ by a line bundle. In particular, this result holds as well for all $V$ with $c_1(V)$ odd. This theorem has the following corollary.

\begin{corollary}\label{kahlertype}
Any K\"ahler manifold $(M',J')$ almost complex equivalent to Tolman's manifold as an almost complex manifold is biholomorphic to the projectivisation of a holomorphic rank two bundle $E$ over $\mathbb CP^2$ with $c_1(E)=-1$, $c_2(E)=-1$.
\end{corollary} 

Thus, to prove Theorem \ref{newmain} we only need to find an appropriate restriction on the K\"ahler cones of projectivisations  of holomorphic bundles $E$ with $c_1(E)=-1$, $c_2(E)=-1$. This is done by exhibiting rational curves on which the form $\OM$ evaluates as $\lambda_2 - 2\lambda_1$ which gives the inequality $\lambda_2<2\lambda_1$.

As for Theorem \ref{main}, its proof is by contradiction in three steps. We assume first that $(\M,\omega)$ admits a compatible K\"ahler metric $g$. Then by Corollary \ref{kahlertype} $\M$ is biholomorphic to the projectivisation of a holomorphic rank two bundle $E$ over $\mathbb CP^2$ with $c_1(E)=-1$, $c_2(E)=-1$. Next we show that in the case that $g$ is invariant under an $S^1$-action as in Theorem \ref{main}, 
there is a smooth $S^1$-invariant rational curve $S$ in $\M$ on which $c_1(\M)$ evaluates negatively. This leads us to a contradiction.


{\bf Acknowledgements.} We would like to thank Valeri Alexeev, Michel Brion, Yoshinori Hashimoto, Maxim Kazarian, John Klein, Silvia Sabatini,  Nick Shepherd-Barron, Constantin Shramov,  Andrew Sommese, Jakub Witaszek for their help. We are particularly grateful to Sasha Kuznetsov for answering numerous questions on holomorphic bundles over $\mathbb CP^2$.  

We would also like to thank Oliver Goertsches, Panagiotis Konstantis and Leopold Zoller for comments on the first version of this paper. We are grateful to the
referee for careful reading and very helpful suggestions improving the presentation.

\section{Preliminary results}

In this section we build up some results about Tolman's manfiold. After setting notations in Section \ref{sec:perliminary}, we recall the construction of Tolman's manifold in Section \ref{sec:tolconstruction}. In Section \ref{circle} we specify the weights of $\mathbb T^2$-action on $\M$, which permits one to calculate the Chern classes and Chern numbers of $\M$. Further, in Section \ref{sec:basic} we single out integer bases of $H_2(\M,\mathbb Z)$ and $H^2(\M,\mathbb Z)$. Finally, in Section \ref{sec:bundles} we recall some basic facts on the topology of projective bundles over $\mathbb CP^2$.

\subsection{Basic facts on Hamiltonian circle actions}\label{sec:perliminary}

Suppose that $(M^{2n},\omega)$ is a compact symplectic manifold. Recall that $(M,\omega)$ has a compatible almost complex structure and the space of such is contractible. Hence, the Chern classes of $TM$ are well defined, we refer to them simply as $c_{i}(M) \in H^{2i}(M,\mathbb{Z})$. 

Now, we recall some definitions and results in the theory of Hamiltonian $S^{1}$-actions (the reader may consult \cite[Section 5]{MS} for more background).

\begin{definition}\normalfont
Suppose $(M,\omega)$ is a $2n$-dimensional symplectic manifold with an effective Hamiltonian $S^{1}$-action.  Let $p$ be a fixed point with weights $\{w_{1},\ldots ,w_{n}\}$. Then the index of $p$ is $$ind(p)=2\cdot \#\{w_{i}:w_{i}<0\}.$$
\end{definition} 

The index is also equal to the index of $H$ as a Morse-Bott function on the critical submanifold containing $p$. The following definition is also standard. 

\begin{definition} \normalfont
Suppose $(M,\omega)$ has an effective Hamiltonian $S^{1}$-action.  

\begin{itemize}

\item Let $M^{S^{1}} \subset M$ denote the set of points fixed by $S^{1}$.

\item  For each integer $k \geq 2$ consider the subgroup $\mathbb{Z}_{k}  \subset S^{1}$ generated by $e^{\frac{2 \pi i}{k}}$. Then define $M^{\mathbb{Z}_{k}}$ to be the set of points in $M$ whose stabiliser contains $\mathbb{Z}_{k}$. 

\item A connected component of $M^{\mathbb{Z}_{k}}$ which is not contained in $M^{S^{1}}$ is called an \emph{isotropy submanifold}. We say that this isotropy submanifold has weight $k$ the stabiliser of its generic point is $\mathbb Z_k$.

\item A $2$-dimensional isotropy submanifold is called an isotropy sphere.

\end{itemize}
\end{definition}

Recall that $M^{\mathbb{Z}_{k}}$ and $M^{S^{1}}$ are possibly disconnected symplectic submanifolds invariant by the $S^{1}$-action \cite[Lemma 5.53]{MS}. Note that the term isotropy sphere is justified: any isotropy submanifold inherits a non-trivial Hamiltonian $S^{1}$-action from $M$, hence if $2$-dimensional it is  diffeomorphic to a $2$-sphere. 

Lastly, we give a standard localisation result for reference.

\begin{lemma}\label{c1formula} Let $M$ be a Hamiltonian $S^1$-manifold and $S\subset M$ be an invariant sphere, which is not point-wise fixed by $S^1$. Let $S_{\max},S_{\min}$ be the fixed points on $S$ where the Hamiltonian attains its maximum and minimum respectively. Denote by $w_i$ the weights at $S_{\max}$ and by $w_i'$ the weights at $S_{\min}$. Then 

$$\langle c_1(M), S\rangle=\frac{\sum w_i'-\sum w_i}{w(S)},$$

where $w(s)$ denotes the order of the stabiliser of a non-fixed point in $S$.
\end{lemma}
\begin{proof}
This follows from ABBV localisation, in particular see \cite[Lemma 2.15]{LP}.
\end{proof}

\subsection{Tolman's manifold}\label{sec:tolconstruction}

In this section we recall the construction of Tolman's manifold $\M$ and state some of its simple properties. The construction is given in \cite[Lemma 4.1]{To1} where Tolman describes a symplectic $6$-manifold $(\M,\omega)$ with a Hamiltonian $\mathbb T^2$-action. We will state this result in a slightly more detailed form, in particular, we will describe a two-parameter family $\OM$ of symplectic forms on $\M$. The symplectic form $\omega$ described in \cite[Lemma 4.1]{To1} is $\omega_{1,2}$ in our notation.

\begin{theorem}[Tolman \cite{To1}]\label{toltheorem} Let $\lambda_1,\lambda_2$  satisfy $0<\lambda_1<\lambda_2$. Then there exists a $6$-dimensional symplectic manifold $(\M,\OM)$ with an effective, Hamiltonian $\mathbb{T}^2$-action, with $6$ isolated fixed points and $9$ invariant symplectic spheres, such that the moment map $\phi: \M\to \mathbb R^2$ has the following properties. 

1) The moment map sends $\M^{S^1}$ to points $(0,0),(0,\lambda_{1}+\lambda_{2}),(\lambda_1,\lambda_1+\lambda_2),(2\lambda_1+\lambda_2,0),(\lambda_2,\lambda_1),(\lambda_1,\lambda_1)$.

2) The collection of nine invariant spheres is sent by $\phi$ to the graph depicted on Figure \ref{fig:TolmanFirst}.
\end{theorem}

\begin{figure}[!ht]
\vspace{-0.7cm}
\begin{center}
\includegraphics[scale=0.55]{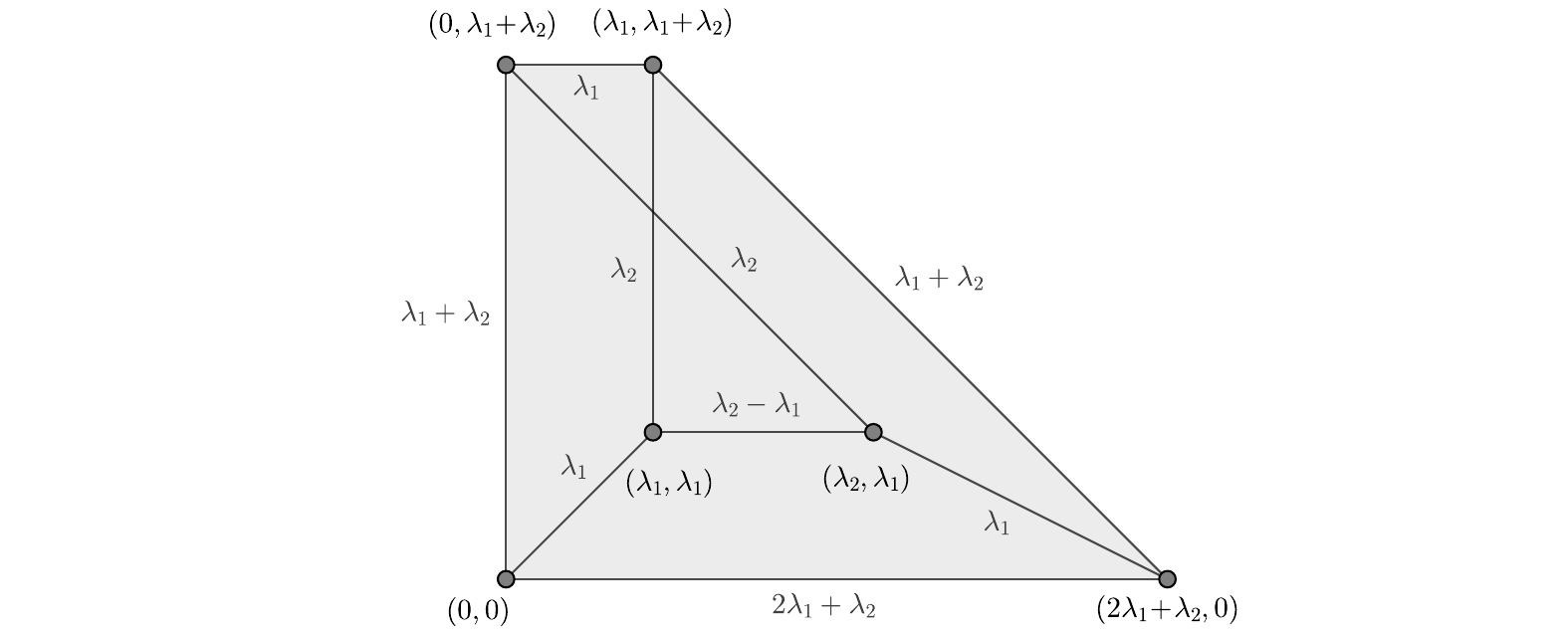}
\end{center}
\vspace{-0.7cm}
\caption{The moment image of $\M$ and the images of the nine invariant spheres}
\vspace{-0.4cm}
\label{fig:TolmanFirst}
\end{figure}
\begin{remark}\normalfont
The moment map $\phi$ is normalised so that the image $\phi(\M)$ lies in the positive quadrant $x,y\ge 0$, and the image of one fixed point is $(0,0)$. The {\it moment image} $\phi(\M)$ is the gray trapezoid depicted on Figure \ref{fig:TolmanFirst}.

The graph with $9$ edges depicted on Figure \ref{fig:TolmanFirst} is sometimes called a {\it GKM-graph}. Each $\mathbb T^2$-invariant sphere in $\M$ is sent by $\phi$ to an edge of the graph. Each such sphere is a symplectic submanifold, and so has positive area, which is indicated next to the corresponding edge. In particular, inequality $\lambda_1<\lambda_2$ ensures that the sphere corresponding to the edge joining points $(\lambda_1,\lambda_1)$ and $(\lambda_2,\lambda_1)$ has positive area. 

We also note that all edges of the graph have a rational slope in $\mathbb R^2$, and the slope doesn't change when $\lambda_1$ and $\lambda_2$ vary.
\end{remark}

{\bf Tolman's construction.} Let us briefly recall the construction of Tolman's manifold from \cite[Section 4]{To1}. We follow the construction very closely, and the only difference is that we take care of the class of symplectic form $\OM$. The form described in \cite[Section 4]{To1} corresponds to choosing $\lambda_1=1$, $\lambda_2=2$.

Tolman starts with two symplectic manifolds $\hat M$ and $\tilde M$, where  $\hat M$ is $\mathbb CP^1\times \mathbb CP^2$ and $\tilde M$ is a submanifold of $\mathbb CP^3\times \mathbb CP^2$ that can be checked to be the projectivisation of the bundle ${\mathcal O}\oplus {\mathcal O}(-3)$ over $\mathbb CP^2$  see \cite[Section 2]{To1}. Both $\hat M$ and $\tilde M$ are toric, so $\mathbb T^3$ naturally acts on them. 

The $\mathbb T^3$-invariant symplectic form on $\hat M$ is chosen so that the  $\mathbb CP^1$-fibre has area $\lambda_1$ and a line in a $\mathbb CP^2$-fibre has area $\lambda_1+\lambda_2$. Such a symplectic form exists for any positive $\lambda_1$ and $\lambda_2$. Stated using the toric terminology, we take the toric $3$-fold, with moment map $\hat \Phi: \hat M \rightarrow \R^3$, such that the moment polytope $\hat\Phi (\hat M)$ has vertices 
$(0,0,0),(\lambda_{1}+\lambda_{2},0,0),(0,\lambda_{1}+\lambda_{2},0),(0,0,\lambda_{1}),(\lambda_{1}+\lambda_{2},0,\lambda_{1}),(0,\lambda_{1}+\lambda_{2},\lambda_{1})$.

For  $\tilde M$ a $\mathbb CP^1$-fibre has area $\lambda_1$ and a line in the projectivisation of the ${\mathcal O}(-3)$-sub-bundle has area $2\lambda_1+\lambda_2$. Stated using toric terminology, we take the toric $3$-fold with moment map $\tilde{\Phi} : \tilde{M} \rightarrow \R^3$, such that polytope $\tilde\Phi (\tilde M)$ has vertices 
$(0,0,0),(2\lambda_{1}+\lambda_{2},0,0),(0,2\lambda_{1}+\lambda_{2},0),(\lambda_{1},\lambda_{1},\lambda_{1}),(\lambda_1,\lambda_{2},\lambda_{1}), (\lambda_2,\lambda_1,\lambda_{1})$. 
Such a $\mathbb T^3$-invariant form exists if and only if $\lambda_2>\lambda_1>0$. Indeed the integral lengths of the top face edges are $\lambda_2-\lambda_1$ and the non-horizontal edges have integral length $\lambda_1$.

\begin{figure}[!ht]
\begin{center}
\includegraphics[scale=0.55]{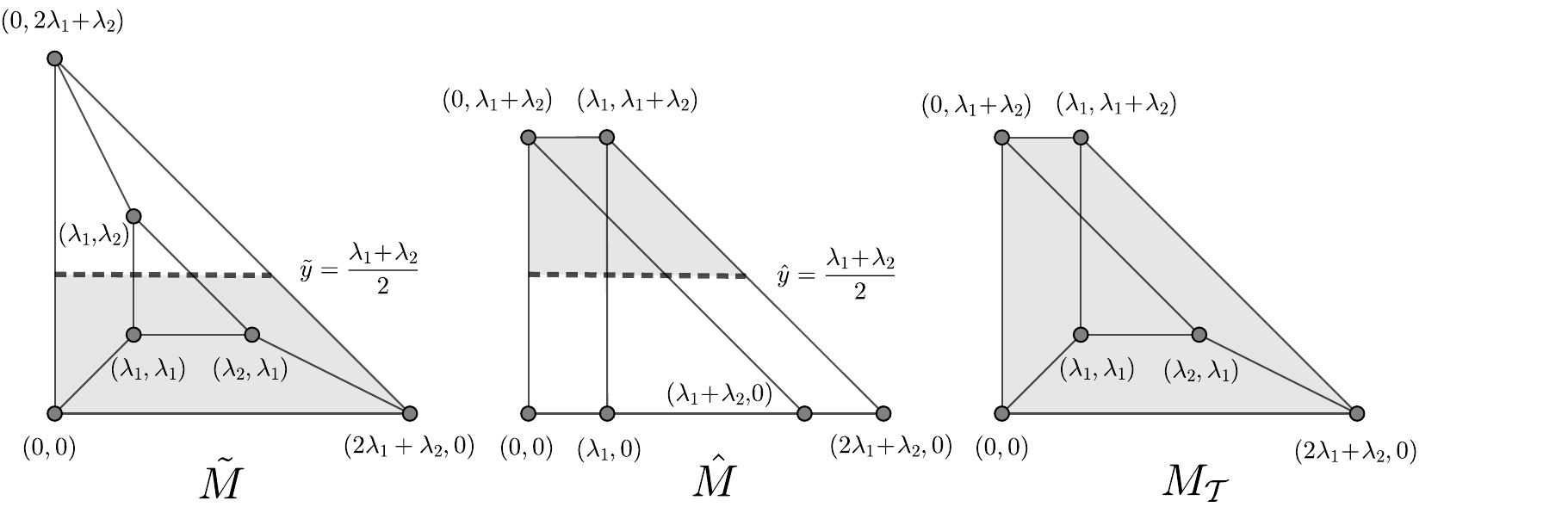}
\end{center}
\vspace{-0.5cm}
\caption{Tolman's sum construction}
\label{fig:TolmanSumI}
\end{figure}

Next one chooses a pair of $2$-dimensional subtori inside $ \mathbb T^3$ with corresponding Hamiltonian functions $\tilde \phi = (\tilde x, \tilde y)$ and $\hat \phi = (\hat x,\hat y)$. The choice of these subtori is given in \cite[Examples 2.2, 2.3]{To1}, we recall them here.  The moment map $\hat \phi$ is defined by composing $\hat{\Phi}$ with the linear map $(x,y,z) \mapsto (x+z,y)$. The moment map $\tilde \phi$ is defined by composing $\tilde{\Phi}$ with the linear map $(x,y,z) \mapsto (x,y)$.

 The images of the corresponding moment maps $\tilde \phi$ and $\hat \phi$ are shown  on the first two pictures of Figure \ref{fig:TolmanSumI}. Both images are contained in the positive quadrant, and one of $6$ fixed points is sent to $(0,0)$. The values of the moment map at all other fixed points are found from the Duistermaat-Heckman formula.

The main step of the construction ensures that two manifolds with boundary given by the inequalities   $\tilde y\le \frac{\lambda_1+\lambda_{2}}{2}$ and $\hat y\ge \frac{\lambda_1+\lambda_2}{2}$ can be glued together along their boundaries at the level  $\frac{\lambda_1+\lambda_2}{2}$ to a symplectic manifold $(\M,\OM)$ satisfying properties of Theorem \ref{toltheorem}. These manifolds are the preimages of the shaded subdomains bounded by the dotted  line. One reason for existence of such a gluing is that the symplectic reductions of $\tilde M$ and $\hat M$ at the level $\tilde y=\hat y=\frac{\lambda_1+\lambda_2}{2}$ by the corresponding $ S^1$-actions are symplectomorphic. In fact, by a theorem of Karshon these $4$-manifolds are $ S^1$-equivariantly symplectomorphic with respect to the remaining $ S^1$-actions corresponding to the Hamiltonians $\tilde x$ and $\hat x$. This finishes our recollection of Tolman's construction, for more details see \cite[Lemma 4.1]{To1}. 

We finish this discussion with the following definition and remark.

\begin{definition}\label{def:tolman}\normalfont Let $\lambda_1,\lambda_2$  satisfy $0<\lambda_1<\lambda_2$. Tolman's manifold $(\M,\OM)$ is  any symplectic $6$-manifold with an effective Hamiltonian $\mathbb T^2$-action that satisfies the conditions of Theorem \ref{toltheorem}.
\end{definition}

We note that by \cite[Corollary 3.2]{GKZ} any two symplectic manifolds satisfying Definition \ref{def:tolman} are diffeomorphic, and moreover such a diffemorphism preserves the homotopy type of a compatible almost complex structure. 
 
\begin{remark}\label{rem:toluinque}\normalfont  Tolman's construction shows that Tolman's manifold exists for any $\lambda_1,\lambda_2$  that satisfy $0<\lambda_1<\lambda_2$. However, it is not immediately obvious that the resulting manifold is unique up to a $\mathbb T^2$-equivariant symplectomorphism. Yet, we believe that for each $0<\lambda_1<\lambda_2$ any two sypmlectic manifolds satisfying the conditions of Definition \ref{def:tolman} are $\mathbb T^2$-equivariantly symplectomorphic. One could try to prove this claim using the fact that the Hamiltonian $S^1$-action induced by the Hamiltonian $\phi^*(y)$ is semi-free. This permits one to apply the classification result of Gonzales \cite{Gon} in the spirit of the work \cite{Cho1, Cho2, Cho3}. We believe that this should allow one to prove that any two such manifolds are at least equivarianlty $S^1$-symplectomorphic.
\end{remark}

%

\subsection{Hamiltonian $S^{1}$-actions on Tolman's manifold}\label{circle}

Let $\phi: \M\to \mathbb R^2$ be the moment map of the $\mathbb T^2$-action. Then for any two comprime integers $a,b$ the pullback function $\phi^*(ax+by)$ is the Hamiltonian of an $S^1$-action on $\M$. The $S^{1}$-action is defined by the subgroup $S^{1} \subset \mathbb{T}^{2}\cong \mathbb{R}^{2}/\mathbb{Z}^{2}$, with parametrization $t \mapsto t(a,b)$. 


We may calculate the weights of the $S^{1}$-action by the following formula: if an oriented edge $E$ of the image of the moment map $\phi(\M)$ (see Figure \ref{fig:TolmanFirst}) is based at the image of a fixed point $p$,  and with direction given by  a primitive vector $V = (x_{1},x_{2})$ with $gcd(x_{1},x_{2})=1$, then a weight at $p$ is given by the formula:

\vspace{-0.3cm}
$$
\omega(V)=bx_2-ax_1 .
$$

The resulting weights of the $S^{1}$-action are given now in Figure \ref{figure}. 

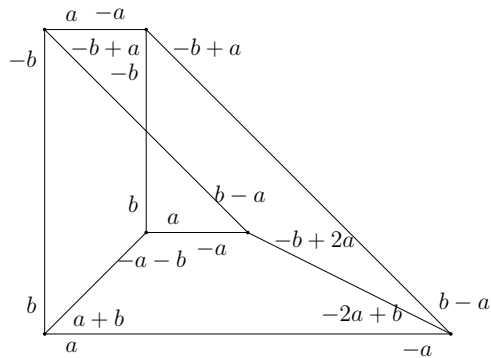
\begin{figure}[H]
\vspace{-0.3cm}
\centering
\scalebox{.75}{
\begin{tikzpicture}[scale=0.45]

\draw[fill] (0,0) circle [radius=0.05];
\draw[fill] (0,12) circle [radius=0.05];
\draw[fill] (16,0) circle [radius=0.05];
\draw[fill] (4,4) circle [radius=0.05];
\draw[fill] (8,4) circle [radius=0.05];
\draw[fill] (4,12) circle [radius=0.05];

 \draw (0,0) --(16,0) -- (4,12) -- (0,12) -- (0,0);

\draw (0,0) -- (4,4) -- (8,4) -- (0,12);

\draw(16,0) -- (8,4);

\draw(4,4) -- (4,12);

\node [above left] at (0,0.5) {$b$};
\node [below right] at (0.5,0) {$a$};
\node [right] at (0.8,0.6) {$a+b$};
\node [below left] at (0,11.5) {$-b$};

\node [above right] at (0.5,12) {$a$};
\node [ right] at (0.7,11.3) {$-b+a$};
\node [above left] at (4,4.5) {$b$};
\node [above right] at (4.5,4) {$a$};
\node [ below] at (4.2,3.6) {$-a-b$};
\node [below left] at (7.5,4) {$-a$};
\node [above] at (7.7,4.8) {$b-a$};
\node [ right] at (8.7,3.7) {$-b+2a$};
\node [right] at (4.7,11.3) {$-b+a$};
\node [above left] at (3.5,12) {$-a$};
\node [below left] at (4,11) {$-b$};
\node [above right] at (15.2,0.6) {$b-a$};
\node [above left] at (14.4,0.1) {$-2a+b$};
\node [below left] at (15.6,0) {$-a$};

\vspace{-0.5cm}
\end{tikzpicture} }

\caption{The weight of the $S^{1}$-action at each of fixed points} \label{figure}
\end{figure}

The following properties of Tolman's manifold also follow from the computation of the weights in Figure \ref{figure}.

\begin{lemma}\label{chernclass} Let $\M$ be Tolman's manifold. Then the following holds.
\begin{enumerate}
\item The value of $c_1(\M)$ on the $\mathbb{T}^{2}$-invariant spheres of $\M$ are $(6,4,4,2,2,2,2,2,0)$ as described on Figure \ref{plain figure} (left). 
\item $\M$ satisfies $c_1(\M)^3=64$.

\item $c_{2}(\M)$ is Poincar\'{e} dual to the sum of the $\mathbb{T}^{2}$-invariant spheres.
\end{enumerate}
\end{lemma}
\begin{figure}[H]
\vspace{-0.5cm}
\centering
\scalebox{.8}{
\begin{tikzpicture}

\draw[fill] (5,0) circle [radius=0.05];
\draw[fill] (5,3) circle [radius=0.05];
\draw[fill] (9,0) circle [radius=0.05];
\draw[fill] (6,1) circle [radius=0.05];
\draw[fill] (7,1) circle [radius=0.05];
\draw[fill] (6,3) circle [radius=0.05];

\node [left] at (5,1.5) {$S_{1}$};

\node [below right] at (5.5,0.7) {$S_{2}$};

\draw (5,0) --(9,0) -- (6,3) -- (5,3) -- (5,0);

\draw (5,0) -- (6,1) -- (7,1) -- (5,3);

\draw (7,1) -- (9,0);

\draw(6,1) -- (6,3);

\draw[fill] (-3,0) circle [radius=0.05];
\draw[fill] (-3,3) circle [radius=0.05];
\draw[fill] (1,0) circle [radius=0.05];
\draw[fill] (-2,1) circle [radius=0.05];
\draw[fill] (-1,1) circle [radius=0.05];
\draw[fill] (-2,3) circle [radius=0.05];

\draw (-3,0) --(1,0) -- (-2,3) -- (-3,3) -- (-3,0);

\draw (-3,0) -- (-2,1) -- (-1,1) -- (-3,3);

\draw(1,0) -- (-1,1);

\draw(-2,1) -- (-2,3);

\node [below] at (-1,0) {$6$};
\node [below] at (-1.5,1) {$0$};
\node [left] at (-3,1.5) {$4$};
\node [above] at (-2.5,3) {$2$};
\node [left] at (-2,1.5) {$2$};
\node [above right] at (-1.7,1.5) {$2$};
\node [above right] at (-0.7,1.5) {$4$};
\node [below right] at (-2.6,0.6) {$2$};
\node [below left] at (-0.2,0.7) {$2$};

\end{tikzpicture} 
}

\caption{Left: The value of $c_{1}(M)$ on the invariant spheres. Right: The image of $\mathbb{T}^{2}$-invariant spheres in Tolman's manifold. } \label{plain figure}
\end{figure}
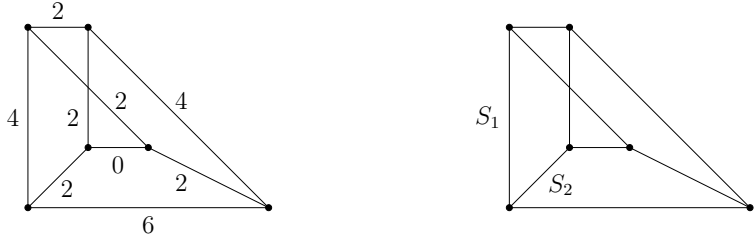

\begin{proof}

1)  The value of $c_1(\M)$ on each $\mathbb{T}^{2}$-invariant sphere can be calculated, using Lemma \ref{c1formula}, from the weights of $S^{1}$-action which are given in Figure \ref{figure}.

2) To find $c_1(\M)^3$ one first chooses on $\M$ an $S^1$-action with isolated fixed points, for example taking $a=2,\,b=1$ above, and calculates the weights at all fixed point by plugging $a=2,\,b=1$ into Figure \ref{figure}. 
Then one applies the ABBV localisation formula \cite{AB,BV}, which is explicitly sated in \cite[Remark 2.5]{To2}. Let us recall this formula. The Chern number $c_1(\M)^3$ is equal to the sum of contributions $F(p)$ over the fixed set $\M^{S^{1}}$, where for a fixed point with weights $w_{1},w_{2},w_{3}$ we have 
$$F(p) = \frac{(w_{1}+w_{2}+w_{3})^3}{w_{1}w_{2}w_{3}},\;\;\; \int_{\M} c_{1}(\M)^3 = \sum_{p \in \M^{S^{1}}}  F(p).$$


3) The Hamiltonian $\mathbb{T}^{2}$-action on $\M$ satisfies the GKM condition by \cite[Lemma 4.1]{GKZ}. By \cite[Lemma 4.11]{GHS} the collection of $\mathbb{T}^{2}$-invariant spheres is a toric $1$-skeleton for $\M$ (in the sense of \cite[Definition 4.10]{GHS}), for any sub $S^{1}$-action, given by a linear subgroup $S^{1} \subset \mathbb{T}^{2}$. Hence  by \cite[Lemma 4.13]{GHS} the sum of $\mathbb{T}^{2}$-invariant spheres is Poincar\'{e} dual to $c_{2}(\M)$.
\end{proof}

\subsection{Integer bases in $H_2(\M,\mathbb Z)$ and $H^2(\M,\mathbb Z)$}\label{sec:basic}

The goal of this section is to prove Corollary \ref{S1S2Basis} and Lemma \ref{baseinH2}, where we construct bases in $H_2(\M,\mathbb Z)$ and $H^2(\M,\mathbb Z)$. We start with the following standard lemma whose proof we present for a lack of a reference.
\begin{lemma}\label{morselemma} Let $M$ be a smooth compact  $2n$-dimensional manifold with a Morse function $f$ whose critical points have even indices. Suppose that for each critical point $x_i$ of index $2m$ there is a $2m$-dimensional compact oriented submanifold $N_i\subset M$ containing $x_i$, such that $f$ restricted to $N$ has a unique maximum that is attained at $x_i$. Then the homology classes $[N_i]$ form an integral basis of $H_{2m}(M,\mathbb Z)$. 
\end{lemma}

For any $t\in \mathbb R$  we denote by $M_{\le t}$ the subset of $M$ consisting of points where $f\le t$ and we denote by $M_{\ge t}$ the subset of points where $f\ge t$. Let us recall the following classical result of Morse theory.

\begin{theorem}\label{standardMorse} Let $M$ be a smooth compact  manifold with a Morse function $f$ whose critical points have even indices. 
\begin{enumerate}
\item For any $t$ and any odd integer $k$, $H_{k}(M_{\le t},\mathbb Z)=0=H_{k}(M_{\ge t},\mathbb Z)$.
\item For any even $k$, $H_k(M,\mathbb Z)$ is a free abelian group of dimension equal to the number of index $k$ critical points of $f$.
\end{enumerate}

\end{theorem}

\begin{proof}[Proof of Lemma \ref{morselemma}]Let $b_{2m}$ be the $2m$-th Betti number of $M$. Thanks to Theorem \ref{standardMorse} 2), to show that the classes $[N_i]$ form an integer basis in $H_{2m}(M,\mathbb Z)$ it is enough to find $b_{2m}$ elements $h_1,\ldots, h_{b_{2m}} $ in $H_{2n-2m}(M,\mathbb Z)$ such that the $b_{2m}\times b_{2m}$ intersection matrix $a_{ij}=[N_i].h_j$ has determinant $\pm 1$.

We will construct the elements $h_j$ as cycles in singular homology.
Let's assume that critical points of index $2m$ are enumerated as $x_1,\ldots, x_{b_{2m}}$,  so that $f(x_i)\le f(x_j)$ for $i<j$.

Let us choose any Riemannian metric $g$ on $M$ and let $U_j$ be the $2n-2m$-dimensional unstable manifold corresponding to $x_j$.  We will define $h_j$ as the class of the following cycle. First choose $\Delta_j$, a $2n-2m$-simplex embedded in $U_j$ such that $x_j$ lies in the interior of $\Delta_j$. Then, by definition of $U_j$, for some $\varepsilon>0$ we have 
$$\min_{\partial(\Delta_j)}f> f(x_j)+\varepsilon.$$
Set $c_j=f(x_j)+\varepsilon$ and note that $\partial \Delta_j$ belongs to $M_{\ge c_j}$. Since $H_{2n-2m-1}(M_{\ge c_j},\mathbb Z)=0$, 
there exists a $2n-2m$-chain $C_j$ contained in  $M_{\ge c_j}$, such that $\partial C_j=\partial(\Delta_j)$. Finally, we set $h_j=[\Delta_j-C_j]$, by construction it is an element of $H_{2n-2m}(M,\mathbb Z)$.

To prove that $\det(a_{ij})=\pm 1$, we note that $(a_{ij})$ is an upper-triangular matrix with entries $\pm 1$ on the diagonal. Indeed, in case $i<j$ the cycles $N_i$ and $\Delta_j-C_j$ are disjoint. And in the case $i=j$ they intersect transversally in one point, namely $x_i$, because by definition $U_i\cap N_i=x_i$ and the intersection at $x_i$ is transversal.
\end{proof}

\begin{corollary}\label{S1S2Basis}  Let $S_{1},S_{2} \subset \M$ be the two $\mathbb{T}^{2}$-invariant spheres depicted in Figure \ref{plain figure} (right). Then $\{[S_{1}],[S_{2}]\}$ is an integral basis of $H_{2}(\M,\mathbb{Z}) \cong \mathbb{Z}^2$.
\end{corollary}
\begin{proof}\footnote{We would like to thank Silvia Sabatini for sharing with us a result on equivariant chohomology that can be used to give an alternative proof of this corollary.} Since for a generic $\mathbb S^1\subset \mathbb T^2$ the corresponding Hamiltonian is a Morse function on $\M$ with exactly two critical points of index $2$, we have by Theorem \ref{standardMorse} 2) that $H_2(\M,\mathbb Z)\cong \mathbb Z^2$. So, according to Lemma \ref{morselemma} it is enough to find a Morse function $f$ on $\M$ so that its restriction to both $S_1$ and $S_2$ has a unique maximum attained at a critical point of $f$ of index $2$. On can check that such a function can be taken as $\phi^*(3x+2y)$, where $\phi$ is the moment map.
\end{proof}

This corollary has the following application.
\begin{lemma}\label{baseinH2} There exist two cohomology classes $\xi',\eta'\in H^2(\M,\mathbb R)$ such that $[\OM]=\lambda_1\xi'+\lambda_2\eta'$. The classes $\xi'$ and $\eta'$ evaluate on invariant spheres as it is shown on Figure \ref{pullback}. Moreover $\xi'$ and $\eta'$ belong to $H^2(\M,\mathbb Z)$ and form its basis.
\end{lemma}

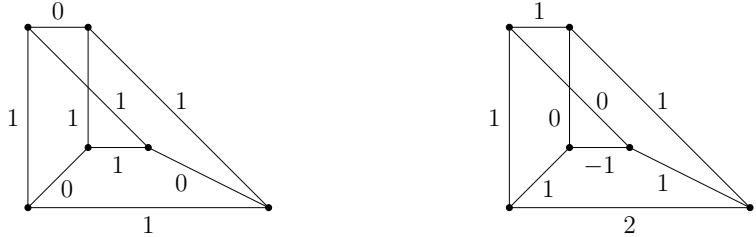
\begin{figure}[H]
\centering
\scalebox{.8}{
\begin{tikzpicture}

\draw[fill] (-3,0) circle [radius=0.05];
\draw[fill] (-3,3) circle [radius=0.05];
\draw[fill] (1,0) circle [radius=0.05];
\draw[fill] (-2,1) circle [radius=0.05];
\draw[fill] (-1,1) circle [radius=0.05];
\draw[fill] (-2,3) circle [radius=0.05];

\node [below] at (-1,0) {$1$};
\node [below] at (-1.5,1) {$1$};
\node [left] at (-3,1.5) {$1$};
\node [above] at (-2.5,3) {$0$};
\node [left] at (-2,1.5) {$1$};
\node [above right] at (-1.7,1.5) {$1$};
\node [above right] at (-0.7,1.5) {$1$};
\node [below right] at (-2.6,0.6) {$0$};
\node [below left] at (-0.2,0.7) {$0$};

\draw (-3,0) --(1,0) -- (-2,3) -- (-3,3) -- (-3,0);

\draw (-3,0) -- (-2,1) -- (-1,1) -- (-3,3);

\draw (-1,1) -- (1,0);

\draw(-2,1) -- (-2,3);

\draw[fill] (5,0) circle [radius=0.05];
\draw[fill] (5,3) circle [radius=0.05];
\draw[fill] (9,0) circle [radius=0.05];
\draw[fill] (6,1) circle [radius=0.05];
\draw[fill] (7,1) circle [radius=0.05];
\draw[fill] (6,3) circle [radius=0.05];

\draw (5,0) --(9,0) -- (6,3) -- (5,3) -- (5,0);

\draw (5,0) -- (6,1) -- (7,1) -- (5,3);

\draw(9,0) -- (7,1);

\draw(6,1) -- (6,3);

\node [below] at (7,0) {$2$};
\node [below] at (6.5,1) {$-1$};
\node [left] at (5,1.5) {$1$};
\node [above] at (5.5,3) {$1$};
\node [left] at (6,1.5) {$0$};
\node [above right] at (6.3,1.5) {$0$};
\node [above right] at (7.3,1.5) {$1$};
\node [below right] at (5.4,0.6) {$1$};
\node [below left] at (7.8,0.7) {$1$};

\end{tikzpicture} 
}

\caption{Left: $\eta'$ on the invariant spheres. Right: $\xi'$ on the invariant spheres.} \label{pullback}
\end{figure}

\begin{proof} The values of $\OM$ on nine $\mathbb T^2$-invariant spheres are given on Figure \ref{fig:TolmanFirst}, and one sees immediately that these values coincide with the values of the functional $\lambda_1\xi'+\lambda_2\eta'$. Since these nine spheres generate $H_2(\M,\mathbb Z)$, we see that $\xi'$ and $\eta'$ are indeed well defined elements of $H^2(\M,\mathbb R)$. 

By Corollary \ref{S1S2Basis}, the classes $[S_1]$ and $[S_2]$ form a basis of $H_2(\M,\mathbb Z)$. It follows from the definition of $\xi'$ and $\eta'$ that classes $\eta'$ and $\xi'-\eta'$ form the basis in $H^2(\M,\mathbb Z)$ dual to the basis $([S_1], [S_2])$. Hence $\xi', \eta'$ is a basis of $H^2(\M,\mathbb Z)$ as well. \end{proof}

\subsection{The topology of projective bundles}\label{sec:bundles}

In this section we recall several results on the topology of $\mathbb CP^1$-bundles over $\mathbb CP^2$. First, we fix some notations.

We will denote by $x$ the positive generator of $H^2(\mathbb CP^2,\mathbb Z)$. For any pair $k_1,k_2\in \mathbb Z$ there is precisely one topological complex rank two bundle $V$ over $\mathbb CP^2$ with $c_{1}(V)=k_1x$ and $c_{2}(V) = k_{2}x^2$,  see \cite[Section 6.1]{OSS}. 

Now, let $\PP(V)$ be the $\mathbb CP^1$-bundle associated to $V$ and let $p:\PP(V) \rightarrow \mathbb CP^2$ be the associated projection. We denote by $\xi$ the class $c_1(\mathcal{O}_{\mathbb{P}(V)}(1)) \in H^2(\PP(V),\Z)$ and by $\eta$ the class $p^*(x) \in H^{2}(\PP(V),\Z)$,  where $\mathcal{O}_{\mathbb{P}(V)}(1)$ is the tautological bundle. The following lemma sums up topological properties of $\PP(V)$,  it is  a rephrasing of \cite[Proposition 15]{OV} in the special case when the base manifold is $\mathbb CP^{2}$.

\begin{lemma} \label{topbundle} Suppose we have a rank $2$ bundle $V$ on $\mathbb CP^2$, and consider the associated projective bundle $\PP(V)$ with projection  $p:\PP(V) \rightarrow \mathbb CP^2$.   
Suppose that $c_{1}(V) = k_{1}x$ and $c_{2}(V) = k_{2}x^2$, then the following holds:
\begin{enumerate}

\item The integer cohomology ring of $\mathbb{P}(V)$ is $\mathbb Z[\eta,\xi]/(\eta^3, \xi^2+ k_{1} \eta \xi+ k_{2} \eta^2)$.

\item  We have the following equalities:
$$c_1(\mathbb{P}(V))= 2\xi+(3+k_1)\eta,\;\; c_1^3(\mathbb{P}(V))=2(27+k_1^2-4k_2).$$

\item $c_1(\mathbb{P}(V))$ is integrally divisible by $2$ if and only if $k_{1}$ is odd. Furthermore, let $J$ be any almost complex structure on $\mathbb{P}(V)$. Then $c_1(\mathbb{P}(V),J)$ is divisible by $2$ if and only if $k_{1}$ is odd.

\item Suppose $k_{1}=-1$. Then for any non-zero  $y\in H^2(\mathbb{P}(V),\mathbb Z)$ one has $y^2\ne 0$. Moreover, if in addition $k_{2}>0$, then for any non-zero $y,\,z\in H^2(\mathbb{P}(V),\mathbb Z)$ one has $y\cdot z\ne 0$.

\item The cubic intersection form on $H^2(\PP(V),\mathbb Z)$  (i.e. the form $F(\beta) = \int_{\PP(V)} \beta^3$) is the following:
$$F(a\eta+b\xi)=b(3a^2-3k_1ab+(k_1^2-k_2)b^2).$$

\end{enumerate}

\end{lemma}

\begin{proof}  

1) This statement follows from Leray-Hirsch Theorem, see the proof of  \cite[Proposition 15]{OV}.

2) Both formulas follow from \cite[Proposition 15]{OV}. The class $c_1(\mathbb{P}(V))$ is calculated from the long exact sequence stated in the proof of  \cite[Proposition 15]{OV}:
\begin{equation}\label{exactotal}
0\to {\cal O}_{\mathbb{P}(V)}\to \pi^* V\otimes{\cal O}_{\mathbb{P}(V)}(1) \to T(\mathbb{P}(V))\to  \pi^* T\mathbb CP^2 \to 0.
\end{equation}

As for $c_1^3(\mathbb{P}(V))$, one applies directly the expression of the cubic form $F_{\mathbb{P}(V)}$ on $H^2(\mathbb{P}(V))$ from \cite[Proposition 15]{OV} to get

$$c_1^3(\mathbb{P}(V))=F_{\mathbb{P}(V)}(c_1(V))=F_{\mathbb{P}(V)}((3+k_{1})\eta+2\xi)=$$
$$2(3(3+k_{1})^2-6k_{1}(3+k_{1})+4(k_1^2-k_2))=2(27+ k_{1}^2 -4k_{2}).$$

3) The first statement follows from the formula for $c_1(\mathbb{P}(V))$ from 2). 

To deduce the second statement from the first one we note that $c_1(\mathbb P(V), J)$ is divisible by $2$ if and only if its reduction modulo $2$ vanishes. At the same time, by  \cite[Proposition 8]{OV}, we have that
$$c_1(\mathbb P(V), J)\mod 2=w_2(\mathbb P(V))=c_1(\mathbb P(V))\mod 2.$$ 
Finally, we have already seen that $c_1(\mathbb P(V))\mod 2=0$ if and only if $k_1$ is odd.

4) We have $c_2(V)=k_{2}x^2$ with $d\in \mathbb Z$. Then to prove the first statement we note that the polynomial $\xi^2-\xi\eta +k_{2} \eta^2$ is not a full square for any $k_{2} \in \mathbb Z$. 

To prove the second statement we note that for $k_{2}>0$ the polynomial $\xi^2-\xi \eta+k_{2}\eta^2$ is irreducible over $\mathbb Q$.

5) This formula is given in \cite[Proposition 15]{OV}.
\end{proof}

From now on, to ease notation, for a vector bundle $V$ over $\mathbb CP^2$ instead of writing $c_1(V)=k_{1}x$ we will write $c_1(V)=k_{1}$ and instead of writing $c_2(V)=k_{2}x^2$ we will write $c_2(V)=k_{2}$. 

Here, we specialise the above lemma to a holomorphic vector bundle $E$ on $\mathbb CP^2$  such that $c_{1}(E)=-1$ and $c_{2}(E)=-1$. In addition we calculate the Chern and Pontryagin classes and characterise the zero-set of the cubic form on $H^{2}(\PP(E),\mathbb{Z})$.

\begin{corollary} \label{tolmantop}
Let  $E$ be a rank $2$ bundle on $\mathbb CP^2$, with $c_{1}=-1$ and $c_{2}=-1$. Let $p : \PP(E) \rightarrow \mathbb CP^2$ be the projection. 
Then the following holds: \begin{enumerate}
\item The cohomology ring of $\PP(E)$ is $\mathbb{Z}[\eta,\xi]/(\eta^{3},\xi^{2} -\eta\xi -\eta^{2})$.

\item The cubic form of $H^2(\PP(E),\mathbb Z)$ (i.e. the form $F(\beta) = \int_{\M} \beta^3$) is 
\begin{equation}\label{cubefirst}
F(a\eta+b\xi) = b(3a^2+3ab+2b^2).
\end{equation} 

\item $c_{1}(\PP(E))=2\eta+2\xi$, $c_{2}(\PP(E)) = 6\xi^2-6\eta^2$ and $p_{1}(\PP(E))) = 8 \eta^2$.

\item Suppose $y \in H^{2}(\PP(E),\Z)$ satisfies $y^3=0$, then $y = k\eta $ for some $k \in \mathbb{Z}$. 

\item $ c_{2}(\PP(E)).\eta =6$,  $c_{2}(\PP(E)).\xi=6$. 
\end{enumerate}
\end{corollary}
\begin{proof}
1) This follows by applying Lemma \ref{topbundle} 1)  to $E$. 

2) This follows by applying Lemma \ref{topbundle} 5)  to $E$.

3) Applying the exact sequence (\ref{exactotal}) to $\PP(E)$ we get
$$c(T\mathbb{P}(E)) = c(p^*(T\mathbb{P}^2)) c(p^*(E) \otimes \mathcal{O}(1)).$$

Note that $ c(p^*(T\mathbb{P}^2)) = 1 + 3\eta + 3\eta^2$. We recall the formula for the total Chern class of a tensor product:  $c(p^*(E) \otimes \mathcal{O}_{\PP(E)}(1)) =  1 + c_{1}(p^*(E)) + 2 c_{1}(\mathcal{O}_{\PP(E)}(1))  + c_{2}(p^*(E)) + c_{1}(p^*(E)) c_{1}(\mathcal{O}_{\PP(E)}(1)) + c_{1}(\mathcal{O}_{\PP(E)}(1)) ^2$. This implies that:
$$ c(p^*(E) \otimes \mathcal{O}_{\PP(E}(1)) = 1 - \eta + 2 \xi -\eta^2 - \eta \xi + \xi^2.$$

So, using the relations in the cohomology ring computed in 1), we obtain $$  c(T\PP(E)) =  1 + (2\xi + 2\eta) + (6 \xi^2 - 6\eta^2) + 6 \xi^2\eta.$$

Hence 
$$c_{1}(\PP(E)) = 2\eta + 2\xi,\;\;c_{2}(\PP(E)) = 6 \xi^2 - 6\eta^2.$$

By  \cite[Proposition 8]{OV}, the first Pontryagin class of an almost complex $6$-manifold satisfies $p_{1}=c_1^2-2c_{2}$, so using the above 
$$p_{1}(\PP(E))=c_1(\PP(E))^2-2c_{2}(\PP(E))=8\eta^2.$$

4) To prove this, we use $2)$ and note that the equation $(3a^2 + 3ab + 2b^2)=0$ has no non-zero real solutions, which means that $F(a\eta+b\xi)\ne 0$ unless $b=0$.

5) This follows from 3) and 1), with a little computation.
\end{proof}

%
%
%
%
%

%
%

\section{Symplectic forms $\omega_{\lambda_{1},\lambda_{2}}$ and topology of $\M$}

In this section we prove several results on the topology of $\M$ needed for the proof of Theorem \ref{newmain}. First, in Section \ref{sec:cubicform} we calculate the cubic intersection form on $H^2(\M,\mathbb Z)$ together with the Chern classes of $\M$. Using this result, in Section \ref{sec:Jupp} one applies a theorem of Jupp, 
to prove that there is a diffeomorphism from $\M$ to $\mathbb P(E)$ that preserves the first Chern class.

\subsection{Chern classes and the cubic intersection form of $\M$}\label{sec:cubicform}

\begin{lemma} \label{othbasis} Let $\M$ be Tolman's manifold, and let $\eta'$ and $\xi'$ be the cohomology classes on $\M$ defined by the evaluation on the invariant spheres as in Lemma \ref{baseinH2} and Figure \ref{pullback}. Then the following holds: \begin{enumerate}
\item $c_{1}(\M) = 2\eta'+ 2 \xi'$. 
\item $c_{2}(\M). \eta' = c_{2}(\M).\xi' = 6.$ \end{enumerate}\end{lemma} \begin{proof}


1) It is enough to check that classes $c_{1}(\M)$ and $2\eta'+ 2 \xi'$ have the same values  on the basis $\{[S_{1}],[S_{2}]\}$. This follows from  Lemma \ref{chernclass}.

2) By Lemma \ref{chernclass}.3), we have that $c_{2}(\M)$ is Poincar\'{e} dual to the sum of the invariant spheres. One may see from Figure \ref{pullback} that summing the value of $\eta'$ and  $\xi'$ over the invariant spheres yields $6$.
\end{proof}

Next, we once again consider a holomorphic vector bundle $E$ over $\PP^2$ with $c_{1}(E)=-1$ and $c_{2}(E)=-1$. We let $p : \PP(E) \rightarrow \mathbb CP^2$ be the projection. Then we recall from Corollary \ref{tolmantop} a natural basis $\{\eta,\xi\}$ for $H^{2}(\M,\mathbb{Z})$, where $\eta = p^{*}(x)$ and $x$ is the hyperplane class of $\mathbb CP^2$ and $\xi$  the first Chern class of the tautological bundle over $\PP(E)$.

\begin{proposition}  \label{intcube}
Let $\omega_{\lambda_{1},\lambda_{2}}$ be the symplectic form constructed in Theorem \ref{toltheorem}, and $\xi',\eta' \in H^{2}(\M,\Z)$ be the integral basis constructed in Lemma \ref{baseinH2}. Then we have that
\begin{equation} \label{cube}  \int_{\M} [\omega_{\lambda_{1},\lambda_{2}}]^3 = \int_{\M} (\lambda_{1} \xi' + \lambda_{2} \eta')^3  =2\lambda_{1}^3 + 3\lambda_{1}^2\lambda_{2} + 3\lambda_{1}\lambda_{2}^2. \end{equation}
\end{proposition}
\begin{proof}
By Lemma \ref{baseinH2}, $\xi'$ and $\eta'$ form an integral basis of $H^{2}(\M,\Z)$. By Theorem \ref{toltheorem} and Lemma \ref{baseinH2}, there is a symplectic form $\omega_{\lambda_{1},\lambda_{2}}$ on $\M$, with cohomology class $\lambda_{1} \xi' + \lambda_{2}\eta'$.
Our first task will be to compute the cubic intersection form on $H^{2}(\M,\R)$. Namely, we will  compute the integral $\int_{\M} [\omega_{\lambda_{1},\lambda_{2}}]^3 = \int_{\M} (\lambda_{1} \xi' + \lambda_{2} \eta')^3 $ using the Duistermaat-Heckman formula  \cite[Theorem 5.55]{MS}. 

Let's take the Hamiltonian $S^{1}$-action on $\M$ corresponding to the linear function $L(x,y) = 2x +y$, i.e. composing the moment map of $\M $ with $L$. The weights for this action are shown in Figure \ref{figure} with $a=2$, $b=1$. Hence, the fixed point data associated to this manifold is as follows:
\bigskip

\begin{tabular}{l*{6}{c}r}
Point              & $P_{1}$ & $P_{2}$ & $P_{3}$ & $P_{4}$ & $P_{5}$ & $P_{6}$ \\
\hline
Co-ordinates & $(0,0)$ &  $(\lambda_1,\lambda_1)$ &  $(\lambda_{2},\lambda_{1})$ &  $(0,\lambda_1+\lambda_2 )$ &  $(\lambda_{1},\lambda_{1}+\lambda_{2})$ &  $(2\lambda_{1}+\lambda_{2},0)$  \\
$H(p_{i})$          & $0$ &$ 3\lambda_1$ &$ \lambda_{1} + 2\lambda_{2}$& $\lambda_1+\lambda_2 $& $3\lambda_{1}+\lambda_{2}$ & $4\lambda_{1}+2\lambda_{2} $\\
$\prod(w_{j})$           & 6 & -6 & 6 & -2 &  2 & -6   \\
\end{tabular}

\bigskip
Here $\prod(w_{j})$ denotes the product of the weights at the fixed point $P_{i}$, and $H$ is the Hamiltonian of the $S^{1}$-action defined above. By the Duistermaat-Heckman formula we have 
$$(-1)^3 \int_{\M} [\omega_{\lambda_{1},\lambda_{2}}]^{3} = \frac{(3\lambda_{1})^3}{-6} + \frac{(\lambda_1+2\lambda_2)^3}{6} + \frac{(\lambda_1+\lambda_2)^3}{-2} + \frac{(3\lambda_1+\lambda_{2})^3}{2} +\frac{(4\lambda_{1} + 2\lambda_{2})^3}{-6} .$$
Simplifying yields the required formula.
\end{proof}

\subsection{The diffeomorphism type of $\M$}\label{sec:Jupp}
In this section we show that $\M$ is diffeomorphic to a projective bundle $\PP(E)$ via a diffeomorphism preserving the almost complex structure, where $E$ is a rank $2$ vector bundle of $\mathbb CP^2$ with $c_{1}(E)=-1$, $c_{2}(E)=-1$. 

The following theorem was proven in \cite{GKZ}. Namely, it was proven in \cite{GKZ}  that Tolman's manifold is diffeomorphic to Eschenburg's manifold. On the other hand,  Eschenburg's manifold is diffeomorphic to the projectivisation of a rank $2$ bundle over $\mathbb CP^2$ \cite[Theorem 2]{E}. Furthermore the vector bundle can be taken with $c_1=1, c_2=-1$  \cite{Ziller}. Note finally, that a projectivisation of a bundle with $c_1=1, c_2=-1$ is diffeomorphic to the projectivisation of a bundle with $c_1=-1, c_2=-1$ since the latter is obtained from the former by tensoring with a line bundle.   We reprove this result calculating the topological invariants in a different way, we believe our approach may find other applications.

\begin{theorem} \cite{GKZ} \label{diffeothm}
Let $\M$ be Tolman's manifold. Then there is a diffeomorphism $\Phi : \PP(E) \rightarrow \M$. Furthermore  $\Phi^{*}(\eta) = \eta'$, $\Phi^{*}(\xi) = \xi'$ and $\Phi$ respects the homotopy class of almost complex structure. 
\end{theorem}

This theorem will be proven using a result of Jupp \cite[Theorem 1.1]{OV}, stated below for the partial case of smooth manifolds with vanishing $b_{3}$. 

\begin{theorem} \label{jupp} \cite[Theorem 1.1]{OV}
Suppose that $M_{1},M_{2}$ are oriented, $1$-connected, closed and smooth $6$-manifolds with torsion free homology and $b_{3}(M_{i}) = 0$ for $i=1,2$. Suppose there is a group isomorphism $Q: H^{2}(M_{2},\Z) \rightarrow H^{2}(M_{1},\Z)$ preserving the following: \begin{enumerate} 

\item The cubic intersection form $F : H^{2}(M_i,\Z) \times H^{2}(M_i,\Z)  \times H^{2}(M_i,\Z) \rightarrow \Z$.
\item The second Steifel-Whitney class $w_{2}(M_{i}) \in H^{2}(M_{i},\Z)/2H^{2}(M_{i},\Z)$.
\item The first Pontryagin class $p_{1}(M_{i}) \in H^{2}(M_{i},\Z)^*$. 
\end{enumerate}

Then, there is a diffeomorphism $\Phi : M_{1} \rightarrow M_{2} $, such that $\Phi^*=Q$. 
\end{theorem}
\begin{proof} This result is proven \cite[Theorem 1.1]{OV}. Indeed, since we only consider smooth manifolds, and our manifolds have a unique smooth structure  (see \cite[Theorem 1.1]{OV}), hence the homeomorphsim guaranteed by \cite[Theorem 1.1]{OV} can be chosen to be a diffeomorphism.
\end{proof}
 
\begin{proof}[Proof of Theorem \ref{diffeothm}]
By Lemma \ref{baseinH2}, $\eta'$ and $\xi'$ are a basis of $H^{2}(M,\mathbb{Z})$. We define a homomorphism 
$$Q: H^{2}(\M,\Z) \rightarrow  H^2(\PP(E),\Z),$$ 
by letting $Q(\eta') = \eta$ and $Q(\xi') = \xi$ and extending  linearly. Note that $Q$ is a group isomorphism, indeed by Corollary \ref{tolmantop}, $\{\eta, \xi\}$ is an integral basis of $H^{2}(\PP(E),\Z)$. The main part of the proof will be to show that $$Q(c_{1}(\M)) = c_{1}(\mathbb{P}(E))  \;\;\; \text{and}  \;\;\;  Q^*(c_{2}(\mathbb{P}(E)) )= c_{2}(\M) ,  $$ where the second equation is interpreted in the dual space of $H^{2}$. Due to Theorem \ref{jupp}, to find the required diffeomorphism $\Phi$ it is sufficient to prove that the isomorphism $Q$ preserves the cubic intersection form, the second Steifel-Whitney class, and the first Pontryagin class. We will prove these in this order.

 The invariance of the Pontryagin class will follow from the invariance of the Chern classes, using the equation $p_1 = c_1^2-2c_2$. Then, finally the fact that $Q(c_{1}(\M)) = c_{1}(\mathbb{P}(E))$ will ensure that the diffeomorphism preserves the homotopy classes of the almost complex structure by \cite[Proposition 8]{OV}.

\textbf{1. Intersection trilinear form.} Comparing Equation (\ref{cube}) from Proposition \ref{intcube} with Equation (\ref{cubefirst}) from Corollary \ref{tolmantop} (2), we have that, for any $\lambda_{1},\lambda_{2} \in \R$, \begin{equation} \int_{\M} (\lambda_{1}\xi' + \lambda_{2} \eta')^3  =   \int_{\PP(E)} (\lambda_{1}\xi + \lambda_{2} \eta)^3. \label{cubicform} \end{equation}

Hence, by the definition of $Q$, for any $\alpha \in H^{2}(M,\mathbb{R})$ we have $$ \int_{\M} \alpha^3 = \int_{\PP(E)} Q(\alpha)^3,$$ i.e. $Q$ preserves the cubic intersection form. Recall that a trilinear from $F$ can be recovered from the associated cubic form $S$, using the polarisation $6F(x,y,z) = S(x+y+z)  - S(x+y) - S(x+z) - S(y+z) + S(x)+S(y)+S(z)$, so $Q$ preserves the intersection trilinear forms.

\textbf{2. Second Steifel-Whitney class.} Both $\M$ and $\PP(E)$ are almost complex manifolds with even $c_{1}$, by Corollary \ref{tolmantop} and Lemma \ref{othbasis}. Hence, both of them have vanishing second Steifel-Whitney class \cite[Proposition 8]{OV}, so in particular $w_{2}$ is preserved by $Q$.

\textbf{3. First Pontryagin Class.} Using the relation  $p_{1} = c_{1}^2-2c_{2}$ for almost complex $6$-manifolds \cite[Proposition 8]{OV}, it is sufficient to show that $Q$ preserves $c_{1}^2$ and $c_{2}$ as dual elements of $H^{2}$. Given that $Q$ preserves the intersection trilinear forms, it is sufficient to show that the intersection of $c_{1}(\M)$ with $\xi',\eta'$ is the same as the inersection of $c_{1}(\PP(E))$ with $\xi,\eta$ respectively. Then it follows that $Q$ preserves $c_{1}^2$ as a dual element since by  Lemma \ref{othbasis} $\{\eta',\xi'\}$ is an integral basis of $H^{2}(\M,\Z)$ and by Corollary \ref{tolmantop} $\{\xi,\eta\}$ is an integral basis of $H^{2}(\PP(E),\Z)$.  These intersection are the same, due to the identities  $c_{1}(\M) = 2\xi' + 2\eta'$  and $c_{1}(\PP(E)) = 2\xi + 2\eta$, proved in Lemma \ref{othbasis} and Corollary \ref{tolmantop} respectively.

To see that $c_{2}$ is preserved, note that by Lemma \ref{othbasis} we have that  $c_{2}(\M).\eta'  =  6$, $c_{2}(\M). \xi' = 6$, and by Corollary \ref{tolmantop} (5), $c_{2}(\PP(E)).\eta  =  6$, $c_{2}(\PP(E)). \xi = 6$ So we see that $c_{2}$ is preserved by $Q$ as a dual element. We have shown that $p_{1}$ is preserved by $Q$.

Hence, by Theorem \ref{jupp} the homomorphism $Q : H^{2}(\M,\Z) \rightarrow  H^2(\PP(E),\Z) $  is induced by a diffeomorphism $\Phi: \PP(E) \rightarrow \M$. Furthermore,  $Q(c_{1}(\M)) = c_{1}(\PP(E))$, by the computations of $c_{1}$ in Corollary \ref{tolmantop} and Lemma \ref{othbasis}. Hence, by \cite[Proposition 8]{OV}, the diffeomorphism respects the homotopy type of almost complex structure. \end{proof}

\subsubsection{Proof of Theorem \ref{twosymplectic} }

\begin{proof}[Proof of the Theorem \ref{twosymplectic}]
As was shown in Theorem \ref{toltheorem}, for each $\lambda_{2}>\lambda_{1}>0$, there is a symplectic form $\omega_{\lambda_{1},\lambda_{2}}$ on $\M$ with cohomology class $\lambda_{1} \xi' + \lambda_{2} \eta'$, having a Hamiltonian $\mathbb{T}^{2}$-action. In Theorem \ref{diffeothm}, it was shown that there is a diffeomorphism $\Phi : \M \rightarrow \PP(E)$, such that $\Phi^{*}(\lambda_{1} \xi + \lambda_{2} \eta) = \lambda_{1} \xi' + \lambda_{2} \eta'$. Via this identification, we get the required symplectic forms on $\PP(E)$.
\end{proof}

\section{Proof of Theorem \ref{newmain}}

In this section we prove the main result of the paper. This is done by proving first Theorem \ref{kahlerbundle} in Section \ref{sec:kahlerbundles}, and then by studying basic properties of holomorphic rank two bundles with $c_1=c_2=-1$ over $\mathbb CP^2$ in Section \ref{sec:c1c2}.
An important idea behind the proof of Theorem \ref{newmain} is that a K\"ahler three-fold with $c_1^3>0$ admits a Mori contraction, and luckily $c_1^3(\M)=64>0$, as we saw in Lemma \ref{chernclass}.

\subsection{K\"ahler metrics on $\mathbb CP^1$-bundles over $\mathbb CP^2$}\label{sec:kahlerbundles}

In this section we prove Theorem \ref{kahlerbundle}, which studies K\"ahler structures on manifolds almost complex equivalent to holomorphic $\mathbb CP^1$-bundles over $\mathbb CP^2$. We start with the following standard statement.

\begin{lemma} \label{btwokahler}
Every compact K\"ahler manifold $Y$ with $b_{2}(Y)=2$ is projective. 
\end{lemma}
\begin{proof}
Since $Y$ K\"{a}hler, we have $b_{2}(Y) = 2 = 2h^{0,2}(Y) + h^{1,1}(Y)$,  and  $h^{1,1}(Y)>0$. Hence $H^{0,2}(Y)=0$, implying that $Y$ has an ample line bundle, and is projective.
\end{proof}

\begin{proof}[Proof of Theorem \ref{kahlerbundle}]
By the assumptions of the theorem and the formula for $c_{1}^3$ given in Lemma \ref{topbundle} 2), we have that $c_1^3(M')>0$. It follows that $K_{M'}$ is not nef\footnote{Otherwise $K_{M'}^3\ge 0$, i.e., $-c_1^3(M')\ge 0$, a contradiction.}.
Recall now that by Lemma \ref{topbundle} 3), $c_1(M')$ is integrally divisible by $2$, and so we can apply \cite[Theorem 10]{ST}. According to this theorem there can be the following four possibilities.

1) There exists a simple blow down $\varphi: M'\to M''$.

2) $M'$ has a structure of an unramified conic bundle $\varphi: M'\to M''$  
over a smooth complex surface $M''$.

3) $M'$ has a structure of a quadric bundle $\varphi: M'\to M''$  
over a smooth curve $M''$.

4) $M'$ is a smooth Fano $3$-fold.

Let us consider separately these four cases.

{\it Case 1.} Let us show that in this case $M''$ is $\mathbb CP^3$ and $M'$ is a simple blow up of $\mathbb CP^3$. This will settle the theorem in this case, since a simple blow up of $\mathbb CP^3$ is the projectivisation of the bundle ${\mathcal O}\oplus {\mathcal O}(-1)$ over $\mathbb CP^2$. 

By our assumption, $\varphi: M'\to M''$ is a simple blow down. Let $y\in H^2(M',\mathbb Z)$ be the pullback of the generator of $H^2(M'',\mathbb Z)$ and let $z$ be the class of the exceptional divisor of the blow down $\varphi$. Clearly, $yz=0\in H^4(M',\mathbb Z)$.  Recall that  $c_1(V)=-1$ and let $c_2=d$.  Applying Lemma  \ref{topbundle} 4), we conclude that $d\le 0$. By Lemma \ref{topbundle} 2) we have $c_1^3(M')=56-8d$, and so $c_1^3(M'')=64-8d$.

 Since by Lemma \ref{btwokahler} $M'$ is projective, $M''$ is a projective $3$-fold with $b_{2}(M'')=1$ and $c_{1}(M'')^{3}>0$, so it is a Fano $3$-fold. All such manifolds apart from $\mathbb CP^3$ have $c_{1}^3$ less than $64$,  we conclude that $d=0$ and $M''$ is $\mathbb CP^3$.

{\it Case 2.} Let us show that in this case  $M'$ is a projectivisation of a rank two vector bundle over $\mathbb CP^2$. We know that $M'$ is an unramified conic bundle over a smooth surface $M''$. Clearly $\pi_{1}(M'')=1$ and $b_2(M'')=1$, so $M''$ has the same Betti numbers as $\mathbb CP^2$. Hence $M''$ is either $\mathbb CP^2$ or a fake projective plane. The latter possibility is excluded since fake projective planes have infinite fundamental group.

It follows that $M'$ is an unramified conic bundle over $\mathbb CP^2$. But each such bundle is a projectivisation of a certain rank two bundle $V'$ because the Brauer group of $\mathbb CP^{2}$ is trivial, see \cite[Proposition 4]{Beauville}. 
Note that $c_1(V')$ is odd, so we can tensor it with a line bundle $L$, so that $c_1(V'\otimes L)=-1$. Then by Lemma \ref{topbundle} 2), $c(V)=c(V')$.


{\it Case 3.} In this case, $M'$ has the structure of a quadric bundle. A generic fibre $F$ of this bundle is a smooth quadric with $[F]^2=0\in H^4(M'',\mathbb Z)$. This contradicts Lemma \ref{topbundle} 4).

{\it Case 4.} By Lemma \ref{topbundle},  $M'$ satisfies the following conditions: $c_{1}$ is even, $c_{1}^{3}$ is divisible by $8$, $b_{3}=0$ and $b_{2}=2$. Upon inspecting the list of Fano $3$-folds \cite{IP}, and noting that a three-fold blown up in a curve cannot have even $c_{1}$, one sees that there are $2$ such examples. These are the projectivisation of the tangent bundle of $\mathbb CP^{2}$ and the projectivisation of the bundle $\mathcal{O} \oplus \mathcal{O}(1)$ over $\mathbb CP^2$. In both cases there exists a rank $2$-bundle $V'$ on $\mathbb CP^2$, with $c_{1}(V')$ odd, such that $M'$ is biholomophic to $\PP(V)$. Finally, as in case $2$, we may normalise $V'$ so that $c_{1}(V')=-1$, and then Lemma \ref{topbundle} 2) ensures that $c(V)=c(V')$.
\end{proof}

\subsection{Rank $2$ vector bundles over $\mathbb CP^2$ with $c_1=-1$, $c_2=-1$ }\label{sec:c1c2}

In this section we prove two simple results on holomorphic bundles over $\mathbb CP^2$ with $c_1=-1$, $c_2=-1$. For a vector bundle $V$ on projective space we set $V(k) = V \otimes \mathcal{O}(k)$.

\begin{lemma}\label{exactseq} Let $E$ be a holomorphic rank $2$ bundle with $c_1=-1, c_2=-1$ over $\mathbb CP^2$. Then $E$ doesn't have line sub-bundles and $H^0(E(-1)) \neq 0$.
\end{lemma}





\begin{proof}
 To see that $E$ doesn't have a line sub-bundle we note that the Chern polynomial $c(E)=1-x-x^2$ has no rational roots and so is not divisible by $1+nx$. On the other hand any rank two bundle $E$ with a line sub-bundle splits topologically and so has $c(E)=(1+nx)(1+mx)$.

Let's now prove that $H^0(E(-1))>0$. Recall that by a theorem of Schwarzenberger \cite{OSS}[Lemma 1.2.7] a rank two bundle over $\mathbb CP^2$ is unstable if $c_1^2-4c_2>0$. Since  $c_1^2(E)-4c_2(E)=5$, we see that $E$ is unstable. Hence by \cite{OSS}[Lemma 1.2.5], $H^{0}(\mathbb CP^2,E) \neq 0$. We let $k$ be the minimal integer such that $H^{0}(E(k)) \neq 0$. Then by \cite[Section 2.2, Paragraph 2]{S}, $E(k)$ has a section vanishing in codimension $2$. We saw that $k \leq 0$, and we claim that $k<0$. Suppose for a contradiction that $k = 0$, then $E$ has a section $s$ vanishing in codimension $2$, but then the length of the scheme $\{s=0\}$ is $c_{2}(E) =-1$, which is absurd. So there exists $k<0$ with $H^{0}(\mathbb CP^2,E(k)) \neq 0$, in particular $H^{0}(\mathbb CP^2,E(-1)) \neq 0$.
\end{proof}


\begin{corollary}\label{badline} Let $E$ be a holomorphic rank $2$ bundle with $c_1=-1, c_2=-1$ over $\mathbb CP^2$. 
\begin{enumerate}
\item There exists $n\ge 2$ and a line $L\subset \mathbb CP^2$ such that the restriction of $E$ to $L$ splits as $\mathcal{O}(n) \oplus \mathcal{O}(-1-n)$.

\item Let $L$ be any line as in 1), let $p^{-1}(L)=D$ be the Hirzebruch surface projecting to $L$, and let $S\subset D$ be the unique smooth curve with negative self-intersection. Then 
\begin{equation}\label{threeintegrals}
\int_Sc_1(T\PP(E))\le -2,\;\;\int_S\eta=1, \;\;\int_S\xi\le -2.
\end{equation}
\end{enumerate}
\end{corollary}
\begin{proof} 1) Using Lemma \ref{exactseq} we can find a non-zero section $s$ of $E(-1)$. Since $c_{2}(E(-1))=1$, the section $s$ has zeros, i.e. $\{s=0\}$ is a non-empty subscheme of $\mathbb CP^2$. Let us take a line $L\subset \mathbb CP^2$ that has a finite non-empty intersection with $\{s=0\}$. Then the section $s$ has a finite number of zeros on $L$, and so the restriction of $E(-1)$ to $L$ contains a sub-bundle $O(n-1)$ for $n\ge 2$. It follows that $E|_L$ contains a line sub-bundle $O(n)$ for $n\ge 2$. 

2) Let $L\subset \mathbb CP^2$ be the line found in 1).  Then the surface $p^{-1}(L) = D$ is isomorphic to a Hirzebruch surface $\mathbb{F}_{m}$, where $m \geq 5$. 

Let $S$ be the section in $D$ with negative self intersection, then we have $S \cdot S=-m\leq -5$. Let us now estimate the integral of $c_1(T\PP(E))$ over $S$. Using the fact that the normal bundles of $D$ and $L$ satisfy the relation $N_{D} \cong p^{*}N_{L}$,  we have that 
$$\int_{S} c_{1}(T\PP(E)) = 3 + S \cdot  S\le -2.$$ 

Next, since $p:S\to L$ is an isomorphism, from  the definition of $\eta$ we have $\int_{S}\eta=1$. Finally, since $c_{1}(T\mathbb P(E))=2\eta+2 \xi$, we conclude $\int_{S}\xi\le -2$. 
\end{proof}

\subsection{Proof of Theorem \ref{newmain} and Corollary \ref{firstquestion}}

We are finally ready to prove the main result.

\begin{proof}[Proof of Theorem \ref{newmain}] Let us assume that $\M$ has a K\"ahler metric $g$ compatible with $\OM$. We will deduce that $\lambda_2>2\lambda_1$.

Recall that by Theorem \ref{diffeothm} there is a diffeomorphism $\Phi:\PP(V)\to \M$ that respects the homotopy type of the almost complex structure, where $V$ is a rank $2$ holomorphic bundle with $c_1(V)=-1$, $c_2(V)=-1$. Since  $4c_2(V)-c_1(V)^2=-5<27$, we can apply Theorem \ref{kahlerbundle} to the K\"ahler manifold $(\PP(V),\Phi^*(g), \Phi^*(\OM))$. According to Theorem \ref{kahlerbundle} this K\"ahler manifold is biholomorphic to the projectivisation of a holomorphic rank two bundle $E$ with $c_1(E)=-1$, $c_2(E)=-1$.

By Corollary \ref{badline} (1), there is a line $L \subset \mathbb CP^2$ over which $E$ splits as $\mathcal{O}(n) \oplus \mathcal{O}(-1-n)$, where $n \geq 2$. Let $S\subset \PP(E)$ be the complex curve given by Corollary \ref{badline} (2). Since $\Phi^*(\OM)$ is K\"ahler, its integral over $S$ is positive, and using Equation (\ref{threeintegrals}) we get 

$$0<\int_S\Phi^*(\OM)=\int_S (\lambda_1\xi+\lambda_2\eta)\le -2\lambda_1+\lambda_2.$$ 

Hence, we have finally that $\lambda_{2}>2\lambda_{1}.$
\end{proof}

\begin{remark} \normalfont
It is a curious coincidence that Theorem \ref{newmain} implies that $(\M,\omega_{1,2})$ does not have any compatible K\"{a}hler metric. The symplectic form $\omega_{1,2}$ is the original symplectic form constructed by Tolman in \cite[Lemma 4.1]{To1}.
\end{remark}

Now, we immediately get a proof of Corollary \ref{firstquestion}.

\begin{proof}[Proof of Corollary \ref{firstquestion}]
We note that the choice of subtorus $S^{1} \subset \mathbb{T}^{2}$ described in the proof of Lemma \ref{chernclass} 2) gives a Hamiltonian $S^{1}$-action with all of the fixed points being isolated on $(\M,\omega_{\lambda_{1},\lambda_{2}})$ for any $0 < \lambda_{1} < \lambda_{2}$. By Theorem \ref{newmain} if $\lambda_{2} \leq 2\lambda_{1}$, $(\M,\omega_{\lambda_{1},\lambda_{2}})$ does not have a compatible K\"{a}hler metric. Making any such choice, for example $\lambda_1=1, \lambda_2=2$, gives the required example. 
\end{proof}

\section{Proof of Theorem \ref{main}}

In this section,  after proving some basic results on coprime $S^1$-actions on $\M$ in Section \ref{sec:coprimetolm}, and recalling basic results on holomorphic $\mathbb C^*$-actions in Section \ref{graphs}, we prove Theorem \ref{main}.

\subsection{Coprime circle actions on Tolman's manifold}\label{sec:coprimetolm}

Recall that a Hamiltonian $S^1$-action on a symplectic manifold is called {\it coprime} is its weights at any fixed point are coprime and have modulus greater than $1$. In this section we study  coprime actions on Tolman's manifold $\M$. Let $\phi: \M\to \mathbb R^2$ be the moment map, then for any $(a,b)\subset \mathbb Z^2$  the Hamiltonian $\phi^*(ax+by)$ defines an $S^1$-action on $\M$. The weights of such an action are shown in Figure \ref{figure}.


It will be convenient for us to introduce some notations for the six fixed points of the $\mathbb T^2$-action. For this we choose $\lambda_1=1$, $\lambda_2=2$ and denote each fixed point as $x_{ij}$, where $(i,j)\subset \mathbb R^2$ is the image of $x_{ij}$  with respect to the moment map $\phi$ corresponding to the symplectic  form $\omega_{1,2}$. Thus the six points are denoted by $x_{00}, x_{40}, x_{11}, x_{21}, x_{03}, x_{13}$ as in Figure \ref{fig:marking}. 

\begin{figure}[!ht]
\vspace{0cm}
\begin{center}
\includegraphics[scale=0.3]{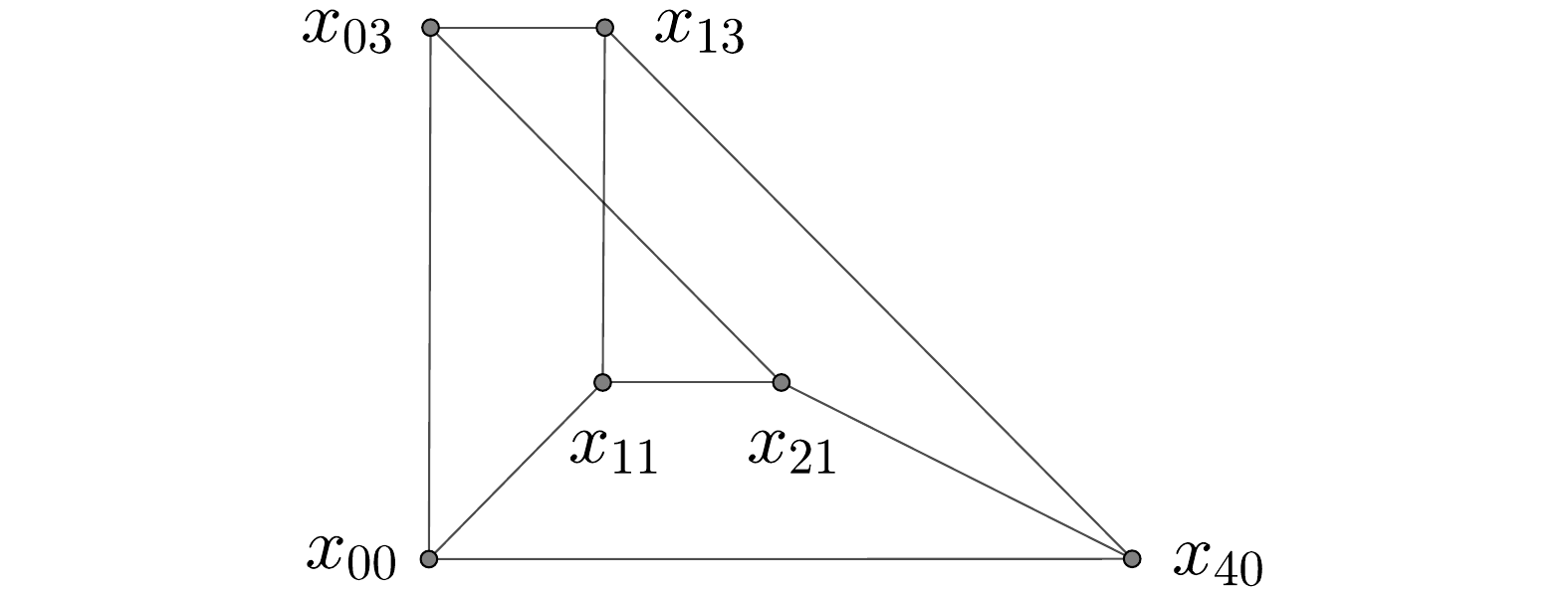}
\end{center}
\vspace{-0.5cm}
\caption{Marking of the fixed points}
\vspace{-0.2cm}
\label{fig:marking}
\end{figure}

%
%
%
%
%

%

\begin{lemma}\label{newcoprime} Let $a$ and $b$ be two non-zero integers. The corresponding $\mathbb S^1$-action on $\M$ is coprime if and only if $a$ and $b$ are coprime, $|a\pm b|>1$, and $|2b-a|>1$. 
\end{lemma}
\begin{proof}
The ``only if'' direction is clear, indeed since $a$ and $b$ are weights at one fixed point they have to be coprime. At the same time the numbers $a\pm b$ and $2b-a$ are the weights of the action at some point and so they are different from $-1,0,1$.

Let us now prove the ``if direction''. From the weights of the $S^1$-action shown on Figure \ref{figure}  it follows that no weight of the circle action is equal to $-1$, $0$, or $1$. Since  $a$ and $b$ are coprime, $a$ and $b$ are each  coprime to both $a+b$ and $a-b$, which implies the weights at $x_{00},x_{11},x_{03},x_{13}$ are pairwise coprime. Note that $-2a+b=-a+(b-a)$ is coprime with both $a$ and $b-a$, hence the weights at $x_{40}$ are pairwise coprime. Similarly, we may check that the weights at $x_{21}$ are pairwise coprime.
\end{proof}


\begin{lemma} \label{4free}
Suppose that $a$ and $b$ are such that the corresponding $S^{1}$-action is coprime. Then the $9$ isotropy spheres for the $S^{1}$-action are exactly the $\mathbb{T}^{2}$-invariant spheres for Tolman's Hamiltonian $\mathbb{T}^{2}$-action. In particular, $c_{1}(\M) \cdot S \geq 0$ for every isotropy sphere $S$.
\end{lemma}

\begin{proof}
Since the action is coprime, $\M$ doesn't have isotropy submanifolds of dimension $4$. 
At the same time the restriction of the $S^{1}$-action to each of the $\mathbb{T}^{2}$-invariant spheres has weight greater than $1$, so these are the isotropy spheres for the action. The last statement now follows immediately from Lemma \ref{chernclass} 1.
\end{proof}

\subsection{Projective manifolds with a $\mathbb C^*$-actions} \label{graphs}

Here state some known results about holomorphic group actions on projective manifolds. Recall that by Lemma \ref{btwokahler} any K\"{a}hler metric on $\M$ is projective, so we will be able to apply these Lemmas to prove Theorem \ref{main}.

The following standard lemma may be found in \cite[Corollary 2.125]{Besse}.
\begin{lemma}\label{holextension} Let $X$ be a K\"ahler manifold with an isometric and Hamiltonian $S^1$-action. Then the $S^1$-action extends to a holomorphic $\mathbb C^*$-action.
\end{lemma}

%
%
%

We will need the following result of Sommese \cite{So}.

\begin{theorem} \label{blanch}
Suppose that $X$ is a complex projective manifold, with an isometric and Hamiltonian $S^{1}$-action.  Then the $S^{1}$-action is the restriction of an algebraic $\C^{*}$-action on $X$.
\end{theorem}
\begin{proof} By Lemma \ref{holextension} the $S^{1}$-action extends to a holomorphic $\mathbb{C}^*$-action. The two main results of \cite{So} imply the action is algebraic. We give some details: Combining \cite[Proposition 1]{So} and  \cite[Proposition 2]{So} shows that the holomorphic map  $F: X \times \mathbb{C}^{*} \rightarrow X$ of the group action, extends to meromorphic map $\tilde{F} : X \times \mathbb{CP}^{1} \rightarrow X$. A  meromorphic map between projective varieties is rational \cite{C}. So in particular $F: X \times \mathbb{C}^{*} \rightarrow X$ is a morphism, i.e. the action is algebraic.
\end{proof}

%

\subsection{Proof of Theorem \ref{main}}

In this section we prove Theorem \ref{main}. For this  we need two more results. In complete analogy with coprime $S^1$-actions, we define  coprime holomorphic $\mathbb C^*$-actions  as actions for which the weights at each fixed point are coprime. 

\begin{lemma}\label{coprimecurve} Let $X$ be a smooth complex projective variety with an algebraic coprime $\mathbb C^*$-action. Suppose that $C\subset X$ is a smooth compact irreducible $\mathbb C^*$-invariant curve. Then $C$ is an isotropy sphere in $X$ with respect to the corresponding $S^1$-action. 
\end{lemma}
\begin{proof} Note that since the $\mathbb C^*$-action is algebraic, $C$ is a compactification of a $\mathbb C^*$-orbit and it has two fixed points. Let $p$ be the fixed point such that the weight at $p$ of the restriction of the $\mathbb{C}^*$-action is positive. Then there exists a small neighbourhood of $p$ where the $\mathbb C^*$-action can be linearised, and it has the form $(z_1,\ldots,z_n)\to (t^{m_1}z_1,\ldots, t^{m_n}z_n)$, where $m_i$ are pairwise coprime weights. So the curve $C$ has a local parametrization $(t^{m_1}c_1,\ldots, t^{m_n}c_n)$ for some collection of complex numbers $c_i$ ($c_i$ has to be zero if $m_i<0$). It follows that the tangent space to $C$ at $(0,\ldots,0)$ has dimension equal to the number of $c_i$'s that are non-zero. In particular it is $1$-dimensional if and only if exactly one of $c_i$'s is non zero, in which case the curve $C$ is an isotropy sphere. This proves the lemma.
\end{proof}

\begin{proposition}\label{badinvariantS} Let $E$ be a holomorphic rank $2$ bundle over $\mathbb CP^2$ with $c_1(E)=-1$, $c_2(E)=-1$. Then for any algebraic $\mathbb C^*$-action on $\mathbb P(E)$ there is a $\mathbb C^*$-invariant smooth rational $S\subset \mathbb P(E)$ such that $\int_{S}c_1(T\mathbb P(E))\le -2$.
\end{proposition}
\begin{proof} Note that the $\mathbb C^*$-action on $\mathbb P(E)$ sends -fibres to fibres, and so it induces a $\mathbb  C^*$-action on $\mathbb CP^2$ due to Blanchards Lemma \cite{Bl}. Let $\mathbb C{P^2}^*$ be the space of projective lines in $\mathbb CP^2$ and $Y\subset \mathbb C{P^2}^*$ be the sub-variety consisting of projective lines $L$ such that the restriction $E|_L$ splits as $\mathcal{O}(n) \oplus \mathcal{O}(-1-n)$, where $n\ge 2$. We know by Corollary \ref{badline} (1) that $Y$ is non-empty. Moreover, $Y$ is a closed analytic subset of $ \mathbb C{P^2}^*$ by \cite[Lemma 3.2.2]{OSS}. Since $Y$ is invariant under the $\mathbb C^*$-action it follows that the action of $\mathbb C^*$ on $Y$ has a fixed point. We conclude that there is a $\mathbb C^*$-invariant line $L$ in $\mathbb CP^2$, such that $E|_L$ is $\mathcal{O}(n) \oplus \mathcal{O}(-1-n)$ where $n\ge 2$.

Now, let $S\subset p^{-1}(L)=D$ be the unique smooth rational curve with negative self-intersection, as in Corollary  \ref{badline} (2). Then by Corollary  \ref{badline} (2) we have $\int_{S}c_1(T\mathbb P(E))\le -2$. At the same time, since $D$ is $\mathbb C^*$-invariant, and $S$ has negative self-intersection in $D$ we see that
$S$ is $\mathbb C^*$-invariant as well. 
\end{proof}

Finally we are ready to prove Theorem \ref{main}.

\begin{proof}[Proof of Theorem \ref{main}] We assume for a contradiction that $\M$ admits a compatible K\"ahler metric invariant under a coprime $S^1$-action. Then, as in the proof of Theorem \ref{newmain}, we get that $\M$ is biholomorphic to the projectivisation of a holomorphic rank two bundle $E$ over $\mathbb CP^2$ with $c_1(E)=-1$, $c_2(E)=-1$. By  Proposition \ref{blanch} the isometric $S^1$-action on $\PP(E)$ extends to an algebraic $\mathbb C^*$-action. So, by Proposition \ref{badinvariantS} there is a smooth $\mathbb C^*$-invariant curve $S$ in $\PP(E)$, such that $\int_{S}c_1(T\PP(E))\le -2$. At the same time, since the $\mathbb C^*$-action is coprime, we know by Lemma \ref{coprimecurve}  that $S$ is an isotropy sphere. This contradicts Lemma \ref{4free}, which states that $c_1(T\PP(E))$ is non-negative on all isotropy spheres in $\M$.
\end{proof}

Nicholas Lindsay, Institute of Mathematical Sciences, ShanghaiTech University No. 393 Huaxia Middle Road, Pudong New Area, Shanghai. lindsaynj@shanghaitech.edu.cn

Dmitri Panov, King's College London, UK.  dmitiri.panov@kcl.ac.uk
  
\end{document}